\documentclass[10pt]{amsart}
\usepackage{geometry}
\usepackage{graphicx}
\usepackage{amssymb}
\usepackage{epstopdf}
\usepackage[colorlinks]{hyperref}
\usepackage[lite]{amsrefs}
\usepackage{mathpazo}
\usepackage{amsmath}
\usepackage{mathrsfs}
\newtheorem{theorem}{Theorem}[section]
\newtheorem{corollary}[theorem]{Corollary}
\newtheorem{lemma}[theorem]{Lemma}

\theoremstyle{definition}

\numberwithin{equation}{section}

\title{Intersection graph and writhe polynomial}
\author{Zhiyun Cheng}
\address{School of Mathematical Sciences, Beijing Normal University, Beijing 100875, China}
\email{czy@bnu.edu.cn}
\begin{document}
\begin{abstract}
We prove that two virtual knots have equivalent intersection graphs if and only if they have the same writhe polynomial.
\end{abstract}
\maketitle
\section{Introduction}\label{section1}
There has been considerable interest in intersection graphs (also called \emph{circle graphs}) of chord diagrams recently. For example, a finite type invariant, also known as Vassiliev invariant, can be regarded as a function on chord diagrams which satisfies the four-term relation and one-term relation. In 1994, Chmutov and Duzhin conjectured that if two chord diagrams have the same intersection graph, then they are equivalent modulo the four-term relation. Together with the fact that the intersection graph is preserved under the mutations of chord diagrams, it follows that all finite type invariants take the same value on mutant knots. However, later it was found that there exists a finite type invariant which distinguishes the Conway knot and Kinoshita-Terasaka knot \cite{MC1996}. Then a question arises, what kind of finite type invariant depends only on the intersection graphs, rather than the chord diagrams? In \cite{CL2007}, Chmutov and Lando proved that a finite type invariant does not distinguish mutants if and only if the function only depends on the intersection graphs of chord diagrams.

Another important feature of finite type invariants comes from the associated Gauss diagrams of knot diagrams. Sometimes, the two terms \emph{Gauss diagram} and \emph{chord diagram} are used indiscriminately. But here, we use the term chord diagram if a chord corresponds to a double point of a singular knot, and when a chord corresponds to a classical crossing point of a knot diagram, we would use the term Gauss diagram. In 1994, Michael Polyak and Oleg Viro gave a description of finite type invariants of order two and three by means of Gauss diagram representations \cite{PV1994}. Later, Mikhail Goussarov \cite{GPV2000} proved that any integer-valued finite type invariant can be described by a Gauss diagram formula. In other words, each integer-valued finite type invariant can be calculated by counting the subdiagrams of the Gauss diagram with weights, where the weight of a subdiagram is the product of the signs of all the chords in the subdiagram. This suggests us to investigate the set of all Gauss diagrams. However, not every Gauss diagram can be realized as the Gauss diagram of a classical knot diagram. This motivates us to study the theory of virtual knots.

Given a virtual knot diagram, there is an associated Gauss diagram and a corresponding intersection graph. Unlike the intersection graphs of finite type invariants, which are undirected simple graphs, an intersection graph derived from a Gauss diagram is a directed graph with signed vertices. Since virtual knots are defined as virtual knot diagrams modulo generalized Reidemeister moves, we need to add some equivalence relations to these intersection graphs. A natural question is, when two virtual knots have equivalent intersection graphs? The main aim of this paper is to answer this question.

\begin{theorem}\label{Theorem1}
Two virtual knots have equivalent intersection graphs if and only if they have the same writhe polynomial.
\end{theorem}

In Section \ref{section2}, we recall the necessary definitions of virtual knots. The equivalence relation on intersection graphs will be defined in Section \ref{section2.1}. A key ingredient, a local move description of the writhe polynomial is given in Section \ref{section2.2}. Section \ref{section3} is devoted to the proof of Theorem \ref{Theorem1}.

\section{Virtual knot theory and writhe polynomial}\label{section2}
\subsection{Virtual knots}\label{section2.1}
Virtual knot theory, which was introduced by L. Kauffman in \cite{Kau1999}, studies the embeddings of circles in thickened surfaces up to isotopies, self-homeomorphisms of the surface and stable equivalences. Virtual knot theory reduces to classical knot theory when the genus of the surface is equal to zero. A convenient way to describe virtual knots is to use virtual knot diagrams.

A virtual knot diagram is simply a classical knot diagram with some classical crossing points replaced by virtual crossing points, where each virtual crossing point is usually denoted by a small circle. Two virtual knot diagrams are equivalent if they are related by a sequence of generalized Reidemeister moves, see Figure \ref{figure1}. A \emph{virtual knot} is defined to be an equivalence class of virtual knot diagrams.

\begin{figure}[h]
\centering
\includegraphics{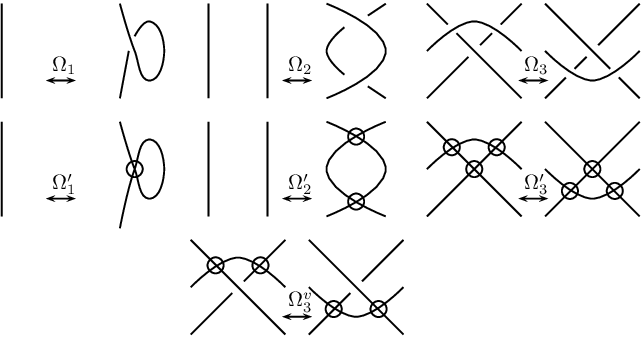}\\
\caption{Generalized Reidemeister moves}\label{figure1}
\end{figure}

Let $K$ be an oriented virtual knot and $D$ a virtual knot diagram of $K$, one can define an associated Gauss diagram $G(D)$ as follows. The Gauss diagram $G(D)$ consists of a planar counterclockwise oriented circle which corresponds to the preimage of the knot diagram, and some signed and directed chords spanning the circle. Here each chord connects the two preimages of a classical crossing point, directed from the preimage of the overcrossing to the preimage of the undercrossing. Finally, each chord is assigned with a sign, which coincides with the sign of the corresponding crossing point.

It is worth pointing out that for a given Gauss diagram, there exist infinitely many different virtual knot diagrams correspond to it. However, all of them correspond to the same virtual knot \cite{GPV2000}.

To a Gauss diagram $G$, the \emph{intersection graph} $\Gamma(G)$ is the graph whose vertices correspond to the chords in $G$ and two vertices are connected by an edge if and only if the corresponding two chords have an intersection point. Each chord can be assigned with a direction as follows. Suppose $c_1$ and $c_2$ are two chords in $G$ such that $c_1\cap c_2\neq\emptyset$, let us use $v_1, v_2$ to denote the two vertices in $\Gamma(G)$ corresponding to $c_1$ and $c_2$ respectively. The edge connecting $v_1, v_2$ is directed from $v_1$ to $v_2$ if locally around the intersection point the orientation $(c_2, c_1)$ coincides with the standard orientation of $\mathbb{R}^2$. Finally, the sign of each vertex in $\Gamma(G)$ is set to be the same as the sign of the corresponding chord in $G$. See Figure \ref{figure2} for an example of the virtual trefoil knot.

\begin{figure}[h]
\centering
\includegraphics{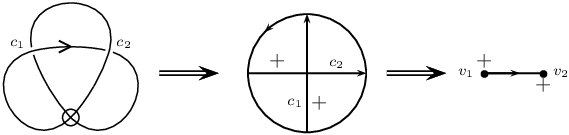}\\
\caption{Virtual trefoil knot, Gauss diagram and intersection graph}\label{figure2}
\end{figure}

Since virtual knot diagrams can be changed by generalized Reidemeister moves, it is necessary to introduce several corresponding local moves on the set of intersection graphs. We say two intersection graphs are \emph{equivalent} if they can be transformed into one another by a sequence of moves $\omega_0, \omega_1, \omega_2, \omega_3$ illustrated in Figure \ref{figure3}.

\begin{figure}[h]
\centering
\includegraphics{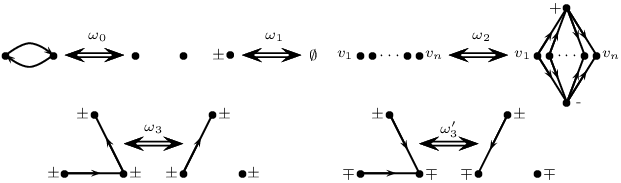}\\
\caption{Local moves on intersection graphs}\label{figure3}
\end{figure}

Here we have some remarks on the local moves depicted in Figure \ref{figure3}. First, the sign of a vertex can be arbitrarily chosen if it is not specified. The reader has recognized that $\omega_1, \omega_2, \omega_3$ are nothing but the deformations on intersection graphs induced by the three classical Reidemeister moves. Note that the vertex on the left side of $\omega_1$ is isolated, and both the positive vertex and negative vertex on the right side of $\omega_2$ have degree $n\geq0$. Reidemeister moves $\Omega_1', \Omega_2', \Omega_3'$ and $\Omega_3^v$ which involve virtual crossing point have no effect on Gauss diagrams, hence preserve the intersection graphs. On the other hand, according to \cite[Theorem 1.2]{Pol2010}, the Reidemeister moves corresponding to $\omega_1, \omega_2, \omega_3$ form a generating set of Reidemeister moves, which guarantees that equivalent virtual knot diagrams have equivalent intersection graphs. For $\omega_0$, which seems a bit odd at first glance, it makes the intersection graph not simple anymore. In order to make it looks reasonable, let us make a few slight modifications to the definition of Gauss diagram. Now we do not require that each chord is a straight segment. Actually, it can be a curved segment inside of the big circle, with no self-intersection. And two Gauss diagrams are considered to be \emph{equivalent} if they are regular homotopy equivalent. Now $\Omega_0$ corresponds to the (dis)appearance of a bigon between two chords.

\subsection{Writhe polynomial}\label{section2.2}
Roughly speaking, there are two different kinds of virtual knot invariants. The first kind comes from the invariants of classical knots. Many classical knot invariants, such as the knot group, knot quandle, Alexander polynomial and Jones polynomial, can be directly defined on virtual knots. Some of them can be greatly enhanced by using the extra information of virtual crossings. The other kind of virtual knot invariants were found later than the introduction of virtual knots. Usually, most of them cannot tell anything interesting for classical knots, but they do provide a lot of useful information for virtual knots. The \emph{writhe polynomial} \cite{CG2013} (also known as the \emph{affine index polynomial} \cite{Kau2013}, see also \cite{Dye2013,Im2013,ST2014}), belongs to the second class.

According to \cite{GPV2000}, a finite type invariant of order one should be the sum of the signs of all chords in the Gauss diagram, which turns out to be the writhe of the knot diagram. Therefore, there exists no finite type invariant of order one for classical knots. However, there do exist finite type invariant of order one for virtual knots. One method to overcome the difficulty is to assign a weight or an index to each chord, then one obtains a weighted sum of all chords, rather than the sum of the signs of all chords.

Let $D$ be a virtual knot diagram and $G$ the corresponding Gauss diagram, choose a chord $c$, we define
\begin{itemize}
\item $r_+(c)$ to be the number of positive chords crossing $c$ from left to right;
\item $r_-(c)$ to be the number of negative chords crossing $c$ from left to right;
\item $l_+(c)$ to be the number of positive chords crossing $c$ from right to left;
\item $l_-(c)$ to be the number of negative chords crossing $c$ from right to left.
\end{itemize}
See Figure \ref{figure4}. Then the \emph{chord index} of $c$ is defined as
\begin{center}
Ind$(c)=r_+(c)-r_-(c)-l_+(c)+l_-(c)$.
\end{center}
By using this, we define the writhe polynomial
\begin{center}
$W_K(t)=\sum\limits_cw(c)t^{\text{Ind}(c)}-w(D)$.
\end{center}
Here $K$ is the virtual knot presented by the knot diagram $D$, $w(c)$ denotes the sign of $c$ and $w(D)=\sum\limits_cw(c)$ is the writhe. The reader is referred to \cite{CFGMX} for some recent progress on index type invariants of virtual knots.

\begin{figure}[h]
\centering
\includegraphics[width=3cm]{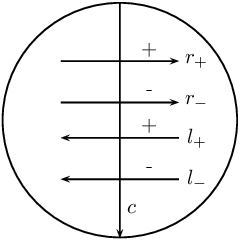}\\
\caption{The definition of the chord index}\label{figure4}
\end{figure}

For a given polynomial invariant, there are two natural questions:
\begin{enumerate}
  \item Which polynomial can be realized as the writhe polynomial of some virtual knot?
  \item When two virtual knots have the same writhe polynomial?
\end{enumerate}
The answer to the first question can be found in \cite{ST2014}, see also \cite[Proposition 4.5]{CFGMX}. It was proved that for a given polynomial $f(t)\in\mathbb{Z}[t, t^{-1}]$, there exists a virtual knot $K$ such that $W_K(t)=f(t)$ if and only if $f(1)=0$ and $f'(1)=0$. Recently, Nakamura, Nakanishi and Satoh gave a local move description of the writhe polynomial \cite{NNS2020}, which provided an answer to the second question.

\begin{theorem}[\cite{NNS2020}]\label{theorem2.1}
Two virtual knots have the same writhe polynomial if and only if they are related by a finite sequence of shell moves.
\end{theorem}

Shell moves are local deformations on Gauss diagrams, which have two different types, say $S_1$ and $S_2$. See Figure \ref{figure5}. Here in the shell move $S_1$, the signs $w(c_1)$ and $w(c_2)$ can be arbitrarily chosen. For the shell move $S_2$, the orientations and the signs of $c_1$ and $c_2$ can be arbitrarily chosen. However, the directions and signs of $c_3$ and $c_4$ should be carefully chosen such that $w(c_3)=-w(c_4)$, Ind$(c_3)=$Ind$(c_4)$, and Ind$(c_1)$, Ind$(c_2)$ are both preserved under the move $S_2$.

\begin{figure}[h]
\centering
\includegraphics{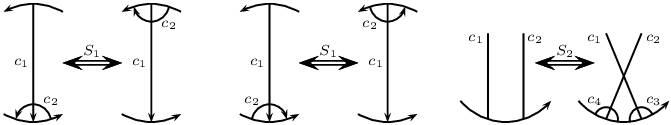}\\
\caption{Shell moves}\label{figure5}
\end{figure}

\section{Proof of Theorem \ref{Theorem1}}\label{section3}
Before giving the proof of Theorem \ref{Theorem1}, we need a simple lemma.
\begin{lemma}\label{lemma3.1}
The two intersection graphs on the two sides of the move $\omega_3'$ in Figure \ref{figure3} are equivalent.
\end{lemma}
\begin{proof}
It is not difficult to find that the local move $\omega_3'$ corresponds to one type of Reidemeister move $\Omega_3$. The result follows immediately since the Reidemeister moves corresponding to $\omega_1, \omega_2, \omega_3$ form a generating set of Reidemeister moves.

More concretely, let we use $v_1, v_2, v_3$ to denote the three vertices on the left side of $\omega_3'$ such that $w(v_1)=+1$ and $w(v_2)=w(v_3)=-1$. Without loss of generality, we assume that in the intersection graph the set of all vertices adjacent to $v_1$ is $\{v_3, v_4, v_5\}$. See the first graph in Figure \ref{figure6}. Now we can apply the move $\omega_2$, which brings us two new vertices with opposite signs. Then after applying the move $\omega_3$, we can use $\omega_2$ to delete two vertices with opposite signs, one of which is $v_1$. Finally, we obtain a new intersection graph isomorphic to the one on the right side of $\omega_3'$. Note that the vertex $v_1$ has been replaced by the new positive vertex $v_1'$.

\begin{figure}[h]
\centering
\includegraphics{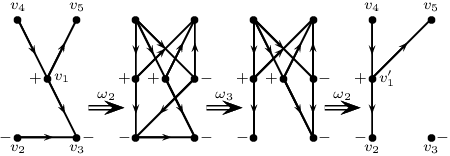}\\
\caption{Generating $\omega_3'$ by $\omega_2$ and $\omega_3$}\label{figure6}
\end{figure}

The other case $w(v_1)=-1$ and $w(v_2)=w(v_3)=+1$ can be proved similarly.
\end{proof}

Now we give the proof of Theorem \ref{Theorem1}.
\begin{proof}
In order to prove that virtual knots with equivalent intersection graphs have the same writhe polynomial, it is sufficient to show that the writhe polynomial can be derived from the intersection graph, and equivalent intersection graphs give the same writhe polynomial. Let $\Gamma$ be an intersection graph and $v$ be a vertex of degree $m$. Let us denote the vertices adjacent to $v$ by $v_1, \cdots, v_n, v_{n+1}, \cdots, v_m$, where the edge $vv_i$ is directed from $v_i$ to $v$ if $1\leq i\leq n$, otherwise it is directed from $v$ to $v_i$. Note that a vertex may appear several times in $\{v_1, \cdots, v_n, v_{n+1}, \cdots, v_m\}$, since multiple edges are allowed. Now we assign an integer Ind$(v)$ to the vertex $v$, defined by
\begin{center}
Ind$(v)=\sum\limits_{1\leq i\leq n}w(v_i)-\sum\limits_{n+1\leq i\leq m}w(v_i)$.
\end{center}
It is easy to see that this is nothing but the index of the chord corresponding to $v$. It follows that the writhe polynomial is equal to $\sum\limits_vw(v)t^{\text{Ind}(v)}-\sum\limits_vw(v)$. It is routine to check that $\sum\limits_vw(v)t^{\text{Ind}(v)}-\sum\limits_vw(v)$ is invariant under the local moves $\omega_0, \omega_1, \omega_2, \omega_3$.

Now suppose we are given two virtual knots $K_1, K_2$ with the same writhe polynomial. Let $D_i$ $(i\in\{1, 2\})$ be a virtual knot diagram of $K_i$ and $\Gamma_i$ be the corresponding intersection graph. What we need to do is to show that $\Gamma_1$ and $\Gamma_2$ are equivalent. According to Theorem \ref{theorem2.1}, it suffices to prove that the corresponding deformations of shell moves on intersection graphs can be realized by $\omega_0, \omega_1, \omega_2, \omega_3$.

For the shell move $S_1$, obviously it preserves the intersection graph. There is nothing need to prove in this case. For the shell move $S_2$, we will discuss each possible case in turn.
\begin{enumerate}
\item Both $c_1, c_2$ are upward directed, and $w(c_1)=w(c_2)=+1, w(c_3)=-w(c_4)=\pm1$. The proof of this case is illustrated in Figure \ref{figure7}. Here we use $v_i$ $(1\leq i\leq 4)$ to denote the corresponding vertex of $c_i$. Note that both the degree of $v_3$ and the degree of $v_4$ are equal to one. The leftmost and rightmost intersection graphs in Figure \ref{figure7} correspond to the Gauss diagram on the right side of $S_2$ in Figure \ref{figure5}, with different choices of $w(c_3)=-w(c_4)$. The middle intersection graph consisting of two positive vertices corresponds to the Gauss diagram on the left side of $S_2$ in Figure \ref{figure5}.
\begin{figure}[h]
\centering
\includegraphics{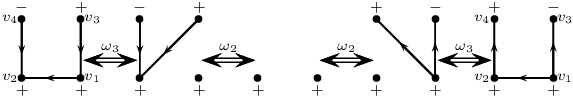}\\
\caption{Realizing the shell move $S_2$ by $\omega_2$ and $\omega_3$}\label{figure7}
\end{figure}

\item Both $c_1, c_2$ are upward directed, and $w(c_1)=w(c_2)=-1, w(c_3)=-w(c_4)=\pm1$. See Figure \ref{figure8} for the proof. As before, the leftmost intersection graph corresponds to the case $w(c_3)=+1, w(c_4)=-1$. The opposite situation corresponds to the rightmost one.
\begin{figure}[h]
\centering
\includegraphics{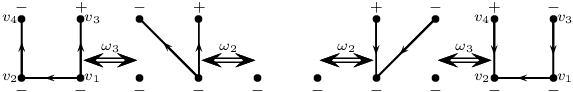}\\
\caption{Realizing the shell move $S_2$ by $\omega_2$ and $\omega_3$}\label{figure8}
\end{figure}

\item Both $c_1, c_2$ are upward directed, and $w(c_1)=-w(c_2)=+1, w(c_3)=-w(c_4)=\pm1$. See Figure \ref{figure9}. Thanks to Lemma \ref{lemma3.1}, here we are allowed to use the move $\omega_3'$.
\begin{figure}[h]
\centering
\includegraphics{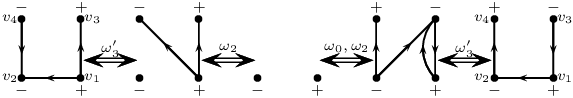}\\
\caption{Realizing the shell move $S_2$ by $\omega_0, \omega_2$ and $\omega_3'$}\label{figure9}
\end{figure}

\item Both $c_1, c_2$ are upward directed, and $w(c_1)=-w(c_2)=-1, w(c_3)=-w(c_4)=\pm1$. The proof of this case is illustrated in Figure \ref{figure10}.
\begin{figure}[h]
\centering
\includegraphics{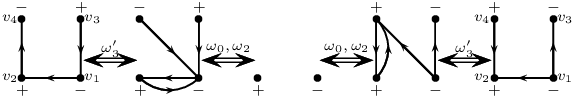}\\
\caption{Realizing the shell move $S_2$ by $\omega_0, \omega_2$ and $\omega_3'$}\label{figure10}
\end{figure}

\item The chord $c_1$ is upward directed but $c_2$ is downward directed, and $w(c_1)=w(c_2)=+1, w(c_3)=-w(c_4)=\pm1$. See Figure \ref{figure11}.
\begin{figure}[h]
\centering
\includegraphics{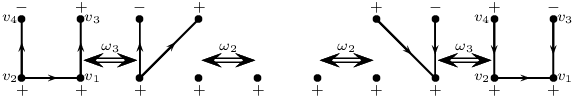}\\
\caption{Realizing the shell move $S_2$ by $\omega_2$ and $\omega_3$}\label{figure11}
\end{figure}

\item The chord $c_1$ is upward directed but $c_2$ is downward directed, and $w(c_1)=w(c_2)=-1, w(c_3)=-w(c_4)=\pm1$. See Figure \ref{figure12}.
\begin{figure}[h]
\centering
\includegraphics{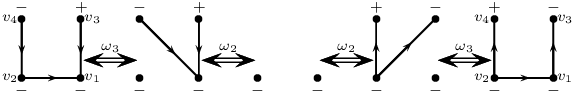}\\
\caption{Realizing the shell move $S_2$ by $\omega_2$ and $\omega_3$}\label{figure12}
\end{figure}

\item The chord $c_1$ is upward directed but $c_2$ is downward directed, and $w(c_1)=-w(c_2)=+1, w(c_3)=-w(c_4)=\pm1$. See Figure \ref{figure13}.
\begin{figure}[h]
\centering
\includegraphics{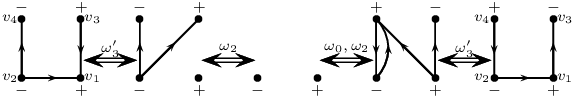}\\
\caption{Realizing the shell move $S_2$ by $\omega_0, \omega_2$ and $\omega_3'$}\label{figure13}
\end{figure}

\item The chord $c_1$ is upward directed but $c_2$ is downward directed, and $w(c_1)=-w(c_2)=-1, w(c_3)=-w(c_4)=\pm1$. See Figure \ref{figure14}.
\begin{figure}[h]
\centering
\includegraphics{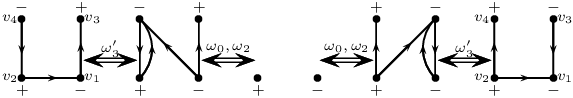}\\
\caption{Realizing the shell move $S_2$ by $\omega_0, \omega_2$ and $\omega_3'$}\label{figure14}
\end{figure}
\end{enumerate}
The case that both $c_1, c_2$ are downward directed and the case that $c_1$ is downward directed but $c_2$ is upward directed can be obtained from the cases discussed above by taking the mirror images. The proof is similar as above, hence we omit it here.
\end{proof}

Theorem \ref{Theorem1} implies that the equivalence classes of intersection graphs of virtual knots contain exactly the same information as the writhe polynomial. Since the characterization problem of the writhe polynomial has been solved in \cite{ST2014}, Theorem \ref{Theorem1} provides a characterization of the intersection graphs of virtual knots up to the local moves $\omega_0, \omega_1, \omega_2, \omega_3$. For each integer $k\geq0$, we define four sequences of intersection graphs $\Gamma(L_k), \Gamma(L_k^{-1}), \Gamma(R_k), \Gamma(R_k^{-1})$, see Figure \ref{figure15}. Then we have the following corollary. We remark that the original characterization problem of the intersection graphs (or circle graphs) was solved by Bouchet in \cite{Bou1994} by a list of excluded vertex minors.
\begin{figure}[h]
\centering
\includegraphics{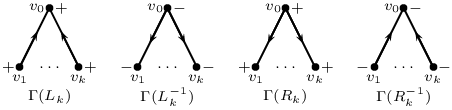}\\
\caption{Four sequences of intersection graphs}\label{figure15}
\end{figure}

\begin{corollary}\label{corollary3.2}
A vertex-signed directed graph can be realized as the intersection graph of a virtual knot diagram if and only if it is equivalent to a disjoint union of some $\Gamma(L_k), \Gamma(L_k^{-1}), \Gamma(R_k)$ and $\Gamma(R_k^{-1})$.
\end{corollary}
\begin{proof}
It was proved in \cite[Proposition 4.5 (4)]{CFGMX} that a polynomial $f(t)\in\mathbb{Z}[t, t^{-1}]$ satisfying $f(1)=0$ and $f'(1)=0$ can be realized as the writhe polynomial of a virtual knot. In particular, this virtual knot can be chosen as the connected sum of some virtual knots $L_k, L_k^{-1}, R_k, R_k^{-1}$, see Figure \ref{figure16}. It is easy to see that the intersection graphs of $L_k, L_k^{-1}, R_k, R_k^{-1}$ are exactly the graphs $\Gamma(L_k), \Gamma(L_k^{-1}), \Gamma(R_k), \Gamma(R_k^{-1})$ defined above. The result follows immediately.
\begin{figure}[h]
\centering
\includegraphics{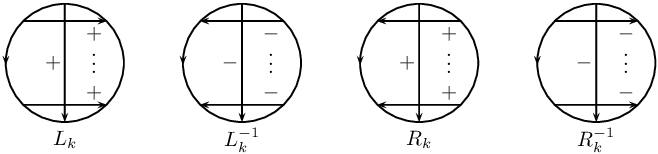}\\
\caption{Four sequences of virtual knots}\label{figure16}
\end{figure}
\end{proof}

We have seen that the intersection graphs of virtual knots contain the same information as the writhe polynomial. It is interesting to ask the same question for some other related knot theory. With a given intersection graph $\Gamma$ and a vertex $v$, the operation \emph{vertex switching} at $v$ changes the sign of $v$ and reverses the orientations of the edges incident to $v$. The corresponding operation on the virtual knot diagram is switching the classical crossing point corresponds to $v$. Flat virtual knots are closely related to virtual knots. Loosely speaking, flat virtual knots can be regarded as virtual knots modulo crossing changes. As a result, if we add the operation vertex switching to the equivalence relations of intersection graphs, then equivalent flat virtual knot diagrams correspond to equivalent intersection graphs. It is interesting to know the accurate information that the intersection graphs contain for flat virtual knots.

\section*{Acknowledgement}
Zhiyun Cheng was supported by the NSFC grant 12071034.


\begin{thebibliography}{0}
\bib{Bou1994}{article}{
	author={Andr\'{e} Bouchet},
	title={Circle Graph Obstructions},
	journal={Journal of Combinatorial Theory, Series B},
	volume={60},
	date={1994},
	number={1},
	pages={107--144}}

\bib{CFGMX}{article}{
	author={Zhiyun Cheng},
	author={Denis Fedoseev},
	author={Hongzhu Gao},
	author={Vassily Manturov},
	author={Mengjian Xu},
	title={From chord parity to chord index},
	journal={Journal of Knot Theory and Its Ramifications},
	volume={29},
	date={2020},
	number={13},
	pages={2043004 (26 pages)}}

\bib{CG2013}{article}{
	author={Zhiyun Cheng},
	author={Hongzhu Gao},
	title={A polynomial invariant of virtual links},
	journal={Journal of Knot Theory and Its Ramifications},
	volume={22},
	date={2013},
	number={12},
	pages={1341002 (33 pages)}}

\bib{CL2007}{article}{
	author={Sergei Chmutov},
	author={Sergei Lando},
	title={Mutant knots and intersection graphs},
	journal={Algebraic\&Geometric Topology},
	volume={7},
	date={2007},
	pages={1579--1598}}

\bib{Dye2013}{article}{
	author={Heather A. Dye},
	title={Vassiliev invariants from Parity Mappings},
	journal={Journal of Knot Theory and Its Ramifications},
	volume={22},
	date={2013},
	number={4},
	pages={1340008 (21 pages)}}

\bib{GPV2000}{article}{
	author={M. Goussarov},
	author={M. Polyak},
	author={O. Viro},
	title={Finite-type invariants of classical and virtual knots},
	journal={Topology},
	volume={39},
	date={2000},
	pages={1045--1068}}

\bib{Im2013}{article}{
	author={Y. H. Im},
	author={S. Kim},
	author={D. S. Lee},
	title={The parity writhe polynomials for virtual knots and flat virtual knots},
	journal={Journal of Knot Theory and Its Ramifications},
	volume={22},
	date={2013},
	number={1},
	pages={1250133 (20 pages)}}

\bib{Kau1999}{article}{
	author={Louis H. Kauffman},
	title={Virtual knot theory},
	journal={Europ. J. Combinatorics},
	volume={20},
	date={1999},
	number={},
	pages={663--691}}

\bib{Kau2013}{article}{
	author={Louis H. Kauffman},
	title={An affine index polynomial invariant of virtual knots},
	journal={Journal of Knot Theory and Its Ramifications},
	volume={22},
	date={2013},
	number={4},
	pages={1340007 (30 pages)}}

\bib{MC1996}{article}{
	author={Hugh Morton},
	author={Peter Cromwell},
	title={Distinguishing mutants by knot polynomials},
	journal={Journal of Knot Theory and Its Ramifications},
	volume={5},
	date={1996},
	number={2},
	pages={225--238}}

\bib{NNS2020}{article}{
	author={T. Nakamura},
	author={Y. Nakanishi},
    author={S. Satoh},
	title={Writhe polynomials and shell moves for virtual knots and links},
	journal={European Journal of Combinatorics},
	volume={84},
	date={2020},
	pages={103033}}

\bib{Pol2010}{article}{
	author={Michael Polyak},
	title={Minimal generating sets of Reidemeister moves},
	journal={Quantum Topology},
	volume={1},
	date={2010},
	pages={399--411}}

\bib{PV1994}{article}{
	author={Michael Polyak},
	author={Oleg Viro},
	title={Gauss diagram formulas for Vassiliev invariants},
	journal={International Mathematics Research Notices},
	volume={11},
	date={1994},
	pages={445--453}}

\bib{ST2014}{article}{
	author={S. Satoh},
	author={K. Taniguchi},
	title={The writhes of a virtual knot},
	journal={Fundamenta Mathematicae},
	volume={225},
	date={2014},
	number={1},
	pages={327--342}}
\end{thebibliography}
\end{document}